\documentclass[a4paper,12pt]{amsart}

\usepackage{graphicx,import}
\usepackage[english]{babel}

\usepackage{url}
\usepackage[pdftitle={Appunti},pdfauthor={Silvio Fanzon, Marcello Ponsiglione and Riccardo Scala}]{hyperref}

\usepackage{mathrsfs}
\hypersetup{colorlinks=false}
\usepackage[disable]{todonotes}
\usepackage[T1]{fontenc}
\usepackage{csquotes}
\usepackage{verbatim}
\usepackage{todonotes}
\usepackage{amsmath}
\usepackage{amsthm}
\usepackage{amsfonts, amssymb, fancybox, multicol, array, tabularx, amscd}
\usepackage{enumerate}
\usepackage{upgreek}
\usepackage[shortlabels]{enumitem}
\usepackage{tikz}
\usepackage{enumitem}
\newtheorem{theorem}{Theorem}[section]
\newtheorem{lemma}[theorem]{Lemma}
\newtheorem{proposition}[theorem]{Proposition}
\newtheorem{corollary}[theorem]{Corollary}

\theoremstyle{definition}

\newtheorem{remark}[theorem]{Remark}

\theoremstyle{remark}

\newcommand{\R}{\mathbb{R}}
\newcommand{\N}{\mathbb{N}}
\newcommand{\Z}{\mathbb{Z}}

\newcommand{\e}{\varepsilon}
\newcommand{\f}{\varphi}

\DeclareRobustCommand{\rchi}{{\mathpalette\irchi\relax}}
\newcommand{\irchi}[2]{\raisebox{\depth}{$#1\chi$}}

\newcommand{\res}{\mathop{\hbox{\vrule height 7pt width .5pt depth 0pt
\vrule height .5pt width 6pt depth 0pt}}\nolimits}

\newcommand{\nor}[1]{\left\| #1 \right\|} %
\newcommand{\norh}[1]{\nor{#1}_{\dot{H}^{\frac 12}}}

\newcommand{\weak}{\rightharpoonup}
\newcommand{\weakstar}{\stackrel{*}{\rightharpoonup}}

\newcommand{\AD}{\mathcal{AD}}

\newcommand{\floor}[1]{\lfloor #1 \rfloor}

\title
[Uniform distribution of dislocations at semi-coherent interfaces]{Uniform distribution of dislocations in  Peierls-Nabarro models for semi-coherent interfaces}

\author[S. Fanzon]
{Silvio Fanzon}
\address[Silvio Fanzon]{(Corresponding author) University of Graz, Institute of Mathematics and Scientific Computing, Heinrichstra\ss e 36, 8010 Graz, Austria}
\email[S. Fanzon]{Silvio.Fanzon@uni-graz.at}

\author[M. Ponsiglione]
{Marcello Ponsiglione}
\address[Marcello Ponsiglione]{Dipartimento di Matematica ``G. Castelnuovo",  Sapienza Universit\`a di Roma,
Piazzale A. Moro 2, 00185 Roma, Italy} \email[M. Ponsiglione]{ponsigli@mat.uniroma1.it}

\author[R. Scala]
{Riccardo Scala}
\address[Riccardo Scala]{Dipartimento di Ingegneria dell'Informazione e Scienze Matematiche, San Niccol\`o, via Roma 56,
53100 Siena, Italy} \email[R. Scala]{riccardo.scala@unisi.it}

\begin{document}
\vskip .2truecm

\begin{abstract}
\small{In this paper we introduce  Peierls-Nabarro type models for edge dislocations at  {\it semi-coherent} interfaces between two heterogeneous crystals, and prove the optimality of uniformly distributed {\it edge dislocations}.
Specifically, we show that the elastic energy $\Gamma$-converges to  a limit functional comprised of two contributions: 
one  is given by a constant $c_\infty>0$ gauging the minimal energy induced by dislocations at the interface, 
and corresponding to a uniform distribution of edge dislocations;
the other one
accounts for the far field elastic energy induced by the presence of further, possibly not uniformly distributed, dislocations. 
After assuming periodic boundary conditions and formally considering the limit from  semi-coherent to coherent interfaces, we show that $c_\infty$ is reached when dislocations are evenly-spaced on the one dimensional circle.

}
\end{abstract}

\maketitle

\vskip -.5truecm
{\small \tableofcontents}

\section*{Introduction}

In this paper we provide a rigorous derivation of  the uniform
distribution 
of dislocations at flat interface boundaries between heterogeneous two-dimensional crystals
whose atomic lattice spacings differ slightly. 
Such configurations are  referred to as {\it semi-coherent} interfaces. 
To this purpose, we propose a  variational model  based on the classical Peierls-Nabarro and Van der Merwe models \cite{N47,P40, VDM}. 
		Our approach consists in minimizing the $H^{\frac 12}$  seminorm \begin{equation}\label{enefra}
	\int_0^l \int_0^l\frac{|u(x)-u(y)|^2}{|x-y|^2} \, dx \, dy\,,
	\end{equation}
	among admissible displacements $u \in W^{1,\infty}(0,l)$, where $l>0$ represents the size of the interface. In this simplified model, the classical Peierls-Nabarro sinusoidal  potential is removed from the energy and replaced by   suitable admissibility conditions on the function $u$: Loosely speaking, $u'$ can assume only two values (determined by the lattice misfit) on intervals whose size is proportional to the lattice spacing.

It is well known  that atomic mismatches are accommodated by  periodic array of dislocations, whose presence decreases the energy of the system. This is common of most of the interface boundaries such as small angle grain boundaries, tilt and twist boundaries \cite{ShRe, VDM}. This is why much effort has been spent in computing the elastic energy induced by periodic distribution of dislocations; we refer to the monograph \cite[Sec 3.3]{NAB} for a comprehensive overview.
A relevant theoretical question is to understand optimal configurations of dislocations without assuming their periodicity:
Rigorous proofs that dislocations are favorable with respect to purely elastic deformations for large interfaces, as well as energy scaling properties 
have been recently faced  in a variety of physical systems related to grain boundaries and epitaxial growth, starting from discrete or semi-discrete models of dislocations \cite{ALP2, ALP,  FPP, FPP2,FFLM,stefanelli,GT, Hirsch, piovano,LL, LPS, mp}. The goal of this paper is to analyze the simplified version of the Peierls-Nabarro model, based on the minimization of the energy in \eqref{enefra}, without assuming any periodicity on $u$. 

We now introduce and derive our model in detail; we consider two lattices $C^+$ and $C^-$, separated by the $x$-axis of $\R^2$, with $C^+$ lying on top of $C^-$. For simplicity we assume that $C^-$ is rigid and in equilibrium, with lattice spacing equals to $\Delta/2$, while $C^+$ is spaced with $\delta/2$, where $0<\delta<\Delta$. We consider the case of a semi-coherent interface, which corresponds to $\delta \approx \Delta$. We analyze the equilibrium conditions for $C^+$, which amounts to study the interfacial displacement $u \colon \frac{\delta}{2} \Z \to \R$, corresponding to the trace of the full strain at the interface. We are interested in a continuous description of this model, hence we consider a suitable affine extension of $u$, imposing that $u' \in \{\lambda, -\Lambda\}$ where $0<\lambda \ll \Lambda$. The region where $u'=\lambda$ corresponds to a purely elastic deformation, yielding a perfect interfacial match between $C^-$ and $C^+$. In contrast $u'=-\Lambda$ describes a dislocation core, corresponding in the atomistic picture to the presence of an extra line of atoms, i.e., an {\it edge dislocation} (see Figure \ref{fig:reference}). For a more precise description of the model we refer to Section \ref{sec:derivation}. 
In this paper %
we assume that the elastic energy is exactly proportional to the 
$H^{\frac12}$ semi-norm of $u$; this is in contrast with the fact that the elastic energy should depend only on the symmetric part of the strain: indeed our choice should be understood as a mere mathematical simplification.

Minimizing \eqref{enefra} among all functions $u$ with $u'\in \{\lambda,-\Lambda\}$, but without any constraint on the size of $\{ u' =-\Lambda\}$, yields oscillating  minimizing sequences converging uniformly to zero and with vanishing energy. 
However, the underlying atomic structure imposes further restrictions on $u$, namely that the regions where $u'=-\Lambda$ have a minimal length, the {\it core radius} $\delta>0$, which is proportional to the lattice spacing. 
Keeping memory of  this important microscopic constraint %
somehow fixes the frequency of oscillations, leading to a well posed minimization problem.   

Since we are interested in the asymptotic behavior of \eqref{enefra} as $l \to +\infty$, we first observe that the minimal energy diverges with order $l$.
Therefore, in order to obtain a meaningful $\Gamma$-convergence result for \eqref{enefra}, we rescale the energy by $l$ and introduce the rescaled functions $w_u\in H^{\frac 12}(0,1)$, defined  by
$$w_u(x):=\frac{u(xl)}{\sqrt{l}}.$$
The energy \eqref{enefra}, scaled by $l$, reads as
\begin{align} \label{enefra2}
 \int_0^1 \int_0^1\frac{|w(x)-w(y)|^2}{|x-y|^2} \, dx \, dy\,,
\end{align}
completed with the constraints described above, depending on $l$,  hidden in the definition of the class of admissible displacements given in \eqref{Dis},  \eqref{wl}.

Our main result is Theorem \ref{mainthm}, which provides the $\Gamma$-limit of the energy functional \eqref{enefra2} as $l\to +\infty$. 
More precisely, the limit functional coincides with
\begin{align}\label{result}
c_\infty +   \int_0^1 \int_0^1\frac{|w(x)-w(y)|^2}{|x-y|^2} \, dx \, dy\,,
\end{align}
where $c_\infty>0$ is a specific constant and $w$ belongs to $H^{\frac 12}(0,1)$, without further constraints.   
In other words, the limit energy splits into two finite contributions: The first, namely the constant $c_\infty$, accounts for the minimal energy induced by the dislocations, which are needed to accommodate the interfacial lattice mismatch. Such dislocations are infinitely many and homogeneously distributed on the interval $(0,1)$ (see Theorem \ref{2.6}). Indeed, the specific value of $c_\infty$ is obtained as the limit as $l\rightarrow+\infty$ of constants $c_l$ defined by %
\begin{align}
 c_l:=\min_u \frac{1}{l} \int_0^l \int_0^l\frac{|u(x)-u(y)|^2}{|x-y|^2} \, dx \, dy\,,
\end{align}
where $u$ is subjected to the usual constraints (see Theorem \ref{thm:cl}). 
The second term in \eqref{result} is induced by the possible presence of further, possibly non uniformly distributed dislocations, inducing a far macroscopic strain. 
It would be desirable to compute
the constant $c_l$ for fixed $l$. However this seems to be out of reach, in particular due to boundary effects at the endpoints of the interval $(0,l)$, which prevent periodic configurations of dislocations. These boundary effects become negligible as $l \to + \infty$; indeed in Theorem \ref{2.6} we show that, for minimizers,  the density of dislocations becomes uniform as $l\rightarrow+\infty$. A further step, which at the present is still missing, would consist in proving real periodicity of the  dislocations  (in the limit as $l\to +\infty$), and therefore that the constant $c_\infty$ agrees with the surface energy density computed in \cite{VDM}.

Finally, we have evidence of the periodic distribution of dislocations for an even more simplified setting, where we assume periodic boundary conditions, fix the number of dislocations, and send $\Lambda\to +\infty$. Roughly speaking, this process  corresponds to consider the limit from semi-coherent to coherent interfaces. More precisely,  in Section \ref{periodicity} we fix $l$ (and thus the number of dislocations) and we recast our functional \eqref{enefra} as defined on $\mathcal{S}^1$ instead of $(0,l)$, thus neglecting the boundary effects. Enforcing the constant $\Lambda\rightarrow+\infty$ and introducing a fixed core-radius $\rho>0$,  we show that the optimal positioning of the dislocations is exactly given by evenly-spaced points on the circle $\mathcal S^1$. This result is proved in Theorem \ref{thm:periodicity}. The elastic energy given by such a configuration corresponds, in this modified setting, 
to the limit of $c_l$ as $l \to +\infty$ and $\frac \delta\Delta \to 0$ simultaneously and in order to keep the  number of dislocations  bounded. 

From a purely mathematical perspective, let us mention that the understanding of periodic configurations as symmetry breaking minimizers of non convex energy functionals is a fascinating and very active research field. A significant impulse to this subject was given by the work in \cite{Mu}. Subsequently, much effort has been devoted in seeking periodic minimizers of Ohta-Kawasaki energy functionals \cite{OK}.  To some extent, our energy can be regarded as a variant of Ohta-Kawasaki energy-type functionals where the $H^{-1}$ norm is replaced by the $H^{-\frac 12} $ norm (see \eqref{OK}), and such a variant was already mentioned as relevant in many respects in \cite{Mu}.%

The energy in \eqref{enefra2} could also be seen as a Modica-Mortola type functional where the Dirichlet term is replaced by the $H^{\frac 12}$ seminorm
	\cite{Alberti,DPV15,FG07, Garroni05, Garroni06, GM12, PV15}.
Here the main difference is that our energy functional formally corresponds to a fractional Modica-Mortola type energy, keeping into account a pre-existing strain, which naturally arises as a consequence of the interfacial lattice mismatch, see \eqref{MoMo}.

While in this paper we enforce the scale of the oscillations by imposing suitable constraints on the class of admissible displacements, in the spirit of the so called {\it core radius approaches}, it would be interesting to study relaxed energies where the length-scale is coerced through scale parameters whose purpose is to tune penalizing potentials; such an approach is more closely related to classical Ohta-Kawasaki and Modica-Mortola energies, and more adherent to the Peierls-Nabarro model. Specifically, the asymptotic behaviour of the functionals introduced in \eqref{MoMo} and \eqref{OK} as $\e\to 0$ deserves, in our opinion, further investigation.

\section{Heuristic derivation and related models} \label{sec:derivation}
In this section we present our model and provide its heuristic derivation from basic semi-discrete models in elasticity. 
\subsection{Semi-coherent interface}
Consider two square lattices $C^\pm$ with different lattice spacing occupying the lower and upper half-plane, respectively. More precisely, let $\tau>0$,  let $0<\delta <\Delta$, and let
$$
C^-:= \frac{\Delta}{2} (\Z \times (-\N)) -(0,\tau), \qquad C^+:= \frac{\delta}{2} (\Z \times \N). 
$$
Here $\frac{\Delta}{2}$ and  $\frac{\delta}{2}$ are the lattice spacing of $C^-$ and $C^+$, respectively, and the convenience of the prefactor $\frac{1}{2}$ will be commented later on (Figure \ref{fig:reference}). We are assuming higher density for the upper crystal $C^+$; nothing would change in our considerations if we assume  instead that  $C^+$ has lower density than $C^-$. Clearly, the case of a single crystal corresponds to $\delta=\Delta=\tau$.

The first mathematical simplification in our model consists in  assuming that $C^-$ is rigid, while  $C^+$ has a  linear elastic behavior. 
Our approach consists in focussing on the position of the atoms on $C^+$ lying on the $x$-axis $\{y=0\}$, which in turn determines the position of all the other atoms in $C^+$ by elastic energy minimization. 
More precisely, we assume that 
each  atom $p$ lying on the $x$-axis,  can move only along the $x$-axis, and positions itself, after being displaced,  either on top of some atom of $C^-$, or in the middle of two adjacent atoms in $C^-$ (see Figure \ref{fig:reference}). The latter case represents an edge dislocation in our model. Moreover, in order to restore the lattice structure where a dislocation is present, we need to somehow complete the dislocation: we assume that two adjacent atoms cannot both be dislocation points; in this way, before and after each dislocation point we have a perfect matching between the two lattices.    We also assume that the deformation preserves orientation and that the distance between two adjacent atoms remains strictly larger than $0$ and smaller than $\frac{\Delta}{2}$. 

It is convenient to describe the deformed configurations of the atoms on the axis $\{y=0\}$ through a displacement function 
$u:\frac{\delta}{2}\Z\to \R$. For what has been said so far, we have that if $p$ and $q$ are two adjacent points  on the $x$-axis,  $p$ to the left of $q$, then $u(q)- u(p)$ can be either $(\Delta -\delta)/2$ or $ \frac{\Delta}{4}- \frac{\delta}{2}$ (see Figure \ref{fig:reference}). If we consider the piece-wise affine extension of $u$ on the whole $\R$, this means that 
$$
u' \in \left\{ \lambda:=   \frac{\Delta -\delta}{\delta}, \, -\Lambda:=  \frac{\Delta}{2 \delta }-1 \right\}.
$$ 
\begin{figure}[t!]
\centering   
\def\svgwidth{14.5cm}   
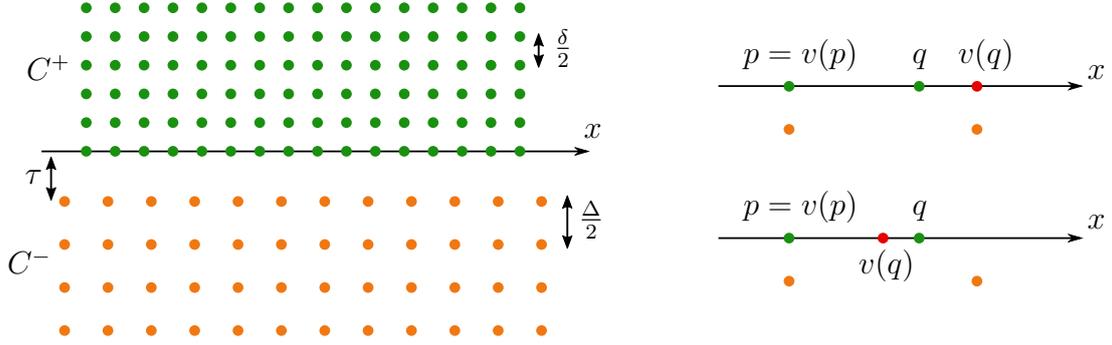 
\caption{Left: reference configuration. Top Right: purely elastic deformation, with relative displacement $(\Delta -\delta)/2$. Bottom Right: deformation leading to an edge dislocation, with relative displacement $\frac{\Delta}{4}- \frac{\delta}{2}$.} 
\label{fig:reference}   
\end{figure}
Assuming that the interface between the two crystals is semi-coherent, namely that $\delta\approx \Delta$, yields $0<\lambda \ll 1$, while $\Lambda \approx \frac 12$.
We also notice that the minimal interval where $u'$ is constant has length equal to $\frac{\delta}{2}$, and in fact $\delta$ if $u'= -\Lambda$. 
In our model we will partially keep memory of this important fact, enforcing that the region where $u'= -\Lambda$ is given by  a  disjoint union of intervals $(x_i - \frac{\delta}{2}, x_i + \frac{\delta}{2})$ of size $\delta$; here the points $x_i$ represent edge dislocations, while $\delta$ is proportional to the lattice spacing, as well as to the size of the Burgers vector and of the core region. In order to relax the stored elastic energy towards its ground state, we expect to deal with zero average displacements $u$; 
this fact, together with assumption $\lambda\ll \Lambda $, yields that  the average distance between dislocations is much larger than $\delta$. In this respect, the information that the region where $u'=\lambda$ is union of intervals of size $\frac{\delta}{2}$ seems to be less relevant, and will be neglected in our model. 
\subsection{The stored elastic energy}
We pass to describe the elastic energy stored in the crystal, in the continuous framework of linearized elasticity. 
Given a displacement $u$ defined on the $x$-axis, one should consider as admissible any displacement $U:\R\times \R^+\to \R^2$ agreeing with $u$ on the $x$-axis. 
Then, %
the stored elastic energy density would be a quadratic form $\langle \mathbb C E_U:E_U\rangle$ where $\mathbb C$
denotes the fourth-order elasticity tensor, and $E_U$ is the symmetrized gradient of $U$.  Instead, here we make the second relevant mathematical simplification of the model (after having assumed $C^-$ rigid):  
we replace the canonical elastic energy density with the simpler  Dirichlet one, given by $| D U|^2$. 
Formally, the energy stored in the crystal is then 
\begin{equation} \label{intro:en_heur}
E(U)= \int_{\R\times \R^+} |D U|^2 \, dz.
\end{equation}
This energy is, in principle, unbounded. However let us ignore this fact for a while and formally minimize \eqref{intro:en_heur} with respect to all $U$ compatible with $u$. This procedure yields the minimal energy induced by $u$, and in turn by the given configuration of dislocations. The minimizer $U_{\rm min}$ has only horizontal component, i.e., $U_{\rm min}=(U_1, 0)$ and, still formally,   
$$
E(U_{\rm min}) = \int_{\R\times \R^+} |D U_1 |^2 \, dz = c \|u\|_{\dot{H}^{1/2}}^2 = c \int_\R \int_\R \frac{|u(x)-u(y)|^2}{|x-y|^2} \, dx \, dy\,,
$$
where $c$ is a suitable pre-factor (see \cite{MP04}). Clearly, in view of our constraints on $u$, such an energy is always infinite. Therefore, we introduce a length-scale $l>0$, and for 
$u:(0,l)\to \R$ define the corresponding elastic energy functionals
$$
E^l(u) = \int_0^l \int_0^l\frac{|u(x)-u(y)|^2}{|x-y|^2} \, dx \, dy\,.
$$
The above energy is finite, and diverges with order $l$ as $l\to +\infty$. Eventually we re-scale $E^l$, multiplying it by $\frac{1}{l}$, and consider its limit in the sense of $\Gamma$-convergence as $l\to + \infty$.

\subsection{Comparison with Peierls-Nabarro models}

In order to establish a comparison between our model and Peierls-Nabarro type models, we need to assume that we are dealing with a single crystal, i.e., $\delta=\Delta$. In such case, we have $\lambda=0$, while $\Lambda = \frac 12$. Let moreover $u:\R\to \R$ be the prescribed displacement function at the $x$-axis. Up to an additive constant, we can always assume that $u(0)=0$. Therefore, the function $u$ is made of affine pieces, where $u'= - \Lambda$, and flat regions, where $u$ takes values in ${\delta}\Lambda \Z = \frac{\delta}{2}\Z$.
Again, the regions where $u'= - \Lambda$ can be identified with the dislocation cores, and $u'\equiv -  \Lambda$ can be understood as a plastic strain. As we will see in Theorem \ref{2.6}, dislocations arise as energy minimizers, and are induced only by the lattice misfit: In this respect, the corresponding strain $u'\equiv -\Lambda$ is usually referred to as {\it eigenstrain}, as it is produced without external forces; we refer the interested reader to \cite{AO}.
    
Here the model is quite rigid, prescribing exactly the values of the strains in the dislocation cores, and a perfect lattice matching $u(x)\in \frac{\delta}{2}\Z$ outside the cores. 
In this respect, the configurations considered in this model are more rigid with respect to the classical Peierls-Nabarro model \cite{P40}, but consistent with their analysis showing that the size of a dislocation is of the order of few lattice spacings. 
Relaxing these conditions gives back the celebrated Peierls-Nabarro model, where the condition $u(x)\in \frac{\delta}{2}\Z$ is enforced by a potential with wells exactly at $\frac{\delta}{2}\Z$. 

In this respect, our model can be regarded as a Peierls-Nabarro model for heterogeneous crystals, whithin the more rigid formalism of eigenstrains.  

On the other hand, one could also consider models which are more closely related to the Peierls-Nabarro formalism. For instance, given $\e>0$ and noticing that $(\lambda + \Lambda)\delta = \frac{\Delta}{2}$, one could consider the functional
$$
E_\e(u) := \e\| u \|^2_{\dot H^{\frac 12}} + \frac{1}{\e} \int_{\R} \text{dist}^2 \Big(u(x) -\lambda x,\frac{\Delta}{2} \Z \Big) \, dx.
$$
Setting $f(x):= u(x) - \lambda x$, the energy can be written as
\begin{equation}\label{MoMo}
E_\e(f) := \e\| f (x) + \lambda x \|^2_{\dot H^{\frac 12}} + \frac{1}{\e} \int_{\R} \text{dist}^2 \Big(f(x),\frac{\Delta}{2} \Z \Big) \, dx.
\end{equation}
Here $\e$ fixes the length-scale of the transitions, so that for $\e\approx \delta$ we expect that the minimizers of the functionals $E_\e$ behave similarly to 
the ones of our proposed model. If on the one hand these energies $E_\e$ seem the natural counterpart of Peierls-Nabarro functionals for heterogeneous crystals, our proposed model seems, at a first glance, more feasible and  easier to be analyzed.    

\subsection{Comparison with the Ohta-Kawasaki model} \label{ohta}
The Ohta-Kawasaki variational model, in its basic form,  consists in minimizing energy functionals acting on scalar functions $u$ and which are given by the sum of the $H^{-1}$ norm of $u$ and a two wells potential forcing $u$ to take values in $\{\pm 1\}$. Such energy is completed with some extra terms, which prescribe the scale at which the oscillations of $u$ occur. 

In fact, it is well known that minimizers exhibit periodic oscillations; the first rigorous proof of this behavior has been carried out in dimension one in \cite{Mu}. There, it is introduced (up to suitable equivalent reformulation) the following energy 
$$
\|v\|_{H^{-1}}^2 +  \e^2 \|v'\|_2^2 + \|v^2-1\|_2^2,
$$  
and it is proven that, for $\e$ small enough, minimizers are periodic. 

Our model is somehow based on minimization of the analogous energy functional where the $H^{-1}$ norm replaced by the  $H^{-\frac 12}$ one.  In fact, consider the energy
$$
\|v\|_{H^{- \frac 12}}^2 +  \e^2 \|v'\|_2^2 + \|v^2-1\|_2^2.
$$  
Setting $u':=v$ in the above expression leads to the minimization of following functional 
$$
\|u\|_{H^{ \frac 12}}^2 +  \e^2 \|u''\|_2^2 + \|(u')^2-1\|_2^2 \, .
$$  
This energy is strictly related to our model for $\lambda = \Lambda =1$. In fact, the third term forces $u'$ to take values in $\{\pm 1\}$, while the second fixes the scale of transitions between such phases, and therefore, together with the first term, determines their number. 
 In our proposed model, the length and number of transitions is enforced replacing the above penalizations with a constraint on the minimal length of the phases. 
 In this respect, it seems interesting to consider the following relaxed version of our energy, more closely related with the Ohta-Kawasaki formalism 
\begin{equation*}
\|u\|_{H^{ \frac 12}}^2 +  \e^2 \|u''\|_2^2 + \|\text{dist} (u', \{\lambda, \, -\Lambda\})\|_2^2 \, ,
\end{equation*}
or, equivalently
\begin{equation}\label{OK}
\|v\|_{H^{- \frac 12}}^2 +  \e^2 \|v'\|_2^2 + \|\text{dist} (v, \{\lambda, \, -\Lambda\})\|_2^2 \, .
\end{equation}

\section{The mathematical model}

\subsection{Admissible configurations and their energy}
Let $\lambda, \Lambda >0$ with $\Lambda \ge \lambda$; let moreover $\delta >0$ be fixed, and $l\ge 0$. We introduce the family of admissible dislocations as

\begin{equation}\label{Dis}
\begin{aligned}
\AD_l := \Big\{ \{ x_1,\, \ldots, \, x_N \} \subset \Big(- \frac{\delta}{2},l + \frac{\delta}{2} \Big) \text{ with } N\in\N, \,   x_1<x_2 \ldots <x_N,  
\\
\text{ and with } \Big(x_i-\frac{\delta}{2}, x_i+\frac{\delta}{2}\Big)    \text{  mutually disjoint} \Big\}\, . 
\end{aligned}
\end{equation}
Notice that dislocations can fall outside of $(0,l)$, and that  the dislocation set can also be empty, namely $\emptyset \in \AD_l$.
Given $X\in \AD_l$, we consider the corresponding displacement at interface $u_X \in  W^{1,\infty}(0,l)$ determined, up to an additive constant,  by
\begin{equation}\label{uxp}
u_X' = \lambda  -(\Lambda + \lambda ) \chi_{(0,l)} \sum_{i=1}^{N} \rchi_{      (x_i-\frac{\delta}{2}, x_i+\frac{\delta}{2})}.
\end{equation}
The class of admissible displacements is given by 
$$
\mathcal U_l:=\{ u_X, \, X\in \AD_l\}\, .
$$
The energy $E^l(u)$ associated to any $u\in \mathcal U_l$ is nothing but the square of its  ${H}^{\frac 12}$-seminorm:
\begin{align} \label{en:h12}
E^l(u) := \|u\|_{\dot{H}^{1/2}}^2 = \int_0^l \int_0^l\frac{|u(x)-u(y)|^2}{|x-y|^2} \, dx \, dy\,.
\end{align}

Clearly, the energy diverges as $l \to + \infty$, therefore we need a suitable rescaling.
We will now provide an estimate for the energy of a function whose oscillation is controlled by a constant. 

\begin{lemma}\label{lemma3.2}
Given $M>0$ there exists $C>0$ such that, for every $l>1$ and for all $u\in \mathcal U_l$ satisfying
\begin{align}\label{estimateLinfty}
 \max\{|u(x)-u(y)|:x,y\in (0,l)\} \le M,
\end{align}
we have
\[
E^l(u) \leq Cl \,.
\]

\end{lemma}

\begin{proof}
Let us set $A:=\{(x,y)\in[0,l]^2:|x-y|\leq 1\}$ and $B:=A^c \cap [0,l]^2$. Thus
\begin{align}\label{integralespezzato}
 E^l(u)=\iint_{A}\frac{|u(x)-u(y)|^2}{|x-y|^2} \, dx\,dy +\iint_{B}\frac{|u(x)-u(y)|^2}{|x-y|^2} \, dx\,dy
\end{align}
and since $u$ is $\Lambda$-Lipschitz we get
\begin{align}
\iint_{A}\frac{|u(x)-u(y)|^2}{|x-y|^2}\, dx\,dy \leq \Lambda^2|A|\leq 2 \,  \Lambda^2l.
\end{align}
Let us estimate the second integral in \eqref{integralespezzato}; using \eqref{estimateLinfty}
we have
\[
\begin{aligned}
&\iint_{B}\frac{|u(x)-u(y)|^2}{|x-y|^2} \, dx\,dy\leq \iint_{B}\frac{M^2}{|x-y|^2} \, dx\,dy = 
2 M^2 \int_1^l\int_0^{x-1} \frac{1}{|x-y|^2}\, dy\,dx \\
&= 2 M^2 (l-1 - \log l)\leq 2 M^2  l
\end{aligned}
\]
for each $l>1$. 
The thesis is achieved with $C:= 2 M^2 + 2 \Lambda^2$.
\end{proof}

\begin{remark}[Energy for evenly spaced dislocations] \label{rem periodic}
We want to show that the energy $E^l$ for a sequence $u_l \in \mathcal{U}_l$ inducing a periodic network of dislocations grows at most linearly in $l$. Set $\gamma:=\frac{\lambda + \Lambda}{\lambda} \, \delta$, $N_l:=\floor{\frac{l}{\gamma}}$ and $r(\gamma,l):=l - N_l \gamma$ so that $0<r(\gamma,l)< \gamma$. The integer $N_l$ represents the number of dislocations which will be present in $(0,l)$. Define intervals
\[ 
I_i := ( (i-1)  \gamma , \, i \gamma - \delta ) \,, \quad J_i := (i \gamma - \delta, \, i \gamma) \,, \quad R:= (N_l \gamma, \, l)
\]
for $i=1,\dots, N_l$, hence obtaining a partition of $(0,l)$. Define $u_l$ as the map such that $u_l(0)=u_l(N_l \gamma) = 0$, $u_l'= \lambda$ on $I_i$ and $R$, $u_l'= -\Lambda$ on $J_i$. 
This is possible since, thanks to the choice of $\gamma$, one can check that $\int_0^{N_l \gamma} u'_l (x) \, dx = 0$. 
In this way $u_l \in \mathcal{U}_l$ and the dislocations are evenly spaced. 

Notice that the maximum oscillation of $u_l$ in $(0,N_l \gamma)$ is exactly $\Lambda\delta$; namely
\begin{align} \label{osc per}
\max\{|u(x)-u(y)|:x,y\in (0,N_l \gamma)\}=\Lambda\delta.
\end{align}
In the interval $R$ the oscillation is given by $\lambda |R|$, and  
\[
\lambda |R| = \lambda \, r(\gamma,l) < \lambda \, \gamma  =  \Lambda \delta + \lambda \delta < 2 \Lambda \delta \,,
\]
thanks to the assumption $\lambda \leq \Lambda$. From \eqref{osc per} we then deduce 
\[
\max\{|u(x)-u(y)|:x,y\in (0,l)\}< 2 \Lambda\delta \,.
\]
Therefore from Lemma \ref{lemma3.2} we conclude that $E^l(u_l)$ scales like $l$. 
\end{remark}

In view of the above remark, we will rescale $E^l$, dividing it by $l$. Exploiting the change of variables $x'=lx$, $y'=ly$ the energy reads
\begin{align}
E^l(u)=\int_0^1 \int_{0}^1 \frac{|u(l x)-u(l y)|^2}{|x-y|^2} \, dx \, dy\,.
\end{align}
We introduce the class of admissible rescaled displacements $\mathcal W_l \subset W^{1,\infty}(0,1)$,
\begin{align}\label{wl}
\mathcal W_l:= \left\{ w_u(z):=\frac{u(lz)}{\sqrt l}, \, u\in \mathcal U_l\right\}.
\end{align}
Setting 
\begin{equation}\label{defl}
F^l(w):=  \int_0^1 \int_{0}^1 \frac{|w(x)-w(y)|^2}{|x-y|^2} \, dx \, dy \qquad \text{ for all } w\in \mathcal W_l\, ,
\end{equation}
the energy can be written as
\[
E^l(u)= l F^l(w_u) \, .
\]

Notice that, given $w\in \mathcal W_l$, there exists $X = X_w\in \AD_l$ and $u_X\in \mathcal U_l$ satisfying \eqref{uxp} such that 
\begin{equation}\label{diw}
w = w_{u_X}.
\end{equation}

\subsection{Asymptotic behavior of the energy functionals}

For each $l>0$ define
\begin{equation} \label{def:cl}
c_l := \min_{u \in \mathcal{U}_l} \frac{1}{l} \int_0^l \int_0^l \frac{|u(x)-u(y)|^2}{|x-y|^2} \, dxdy \, .
\end{equation}

\begin{theorem} \label{thm:cl}
There exists $0<c_\infty< +\infty$ such that 
\begin{equation}\label{thm:cl_c}
\lim_{l \to +\infty } c_l = c_\infty \,.
\end{equation}
	
\end{theorem}

\begin{proof}
First notice that there exists $C>0$ such that
\[
\sup_{l>1} c_l \leq C \,.
\] 
This follows by choosing the maps $u_l$ defined in 
Remark \ref{rem periodic} as competitors for $c_l$. 

Now, for each $l$, let $u_l \in \mathcal{U}_l$ be a minimizer for $c_l$ defined in \eqref{def:cl} (whose existence follows by the standard direct method). Fix $h>0$. For each $l > h$ we define intervals $I_i := (h (i-1), h i)$ for $i=1,\dots,N$, where $N:=\floor{l/h}$. Define the remainder $r(h,l):=l- h N$ and notice that $0 \leq r(h,l) < h$.  We can choose an interval $I_j$ such that
	\[
	\int_{I_j} \int_{I_j} \frac{| u_l(x)  - u_l(y)  |^2}{|x-y|^2} \, dx \, dy \leq 
	\frac{1}{N} \sum_{i=1}^N \int_{I_i} \int_{I_i} \frac{| u_l(x)  - u_l(y)  |^2}{|x-y|^2} \, dx \, dy \,.
		\] 
		Since $u_l$ is a competitor for $c_h$ on $I_j$, we can estimate
		\begin{equation} \label{thm:cl:1}
		\begin{aligned}
		c_h & \leq \frac{1}{h}  \int_{I_j} \int_{I_j} \frac{| u_l(x)  - u_l(y)  |^2}{|x-y|^2} \, dx \, dy \leq  \frac{1}{h N} \sum_{i=1}^N \int_{I_i} \int_{I_i} \frac{| u_l(x)  - u_l(y)  |^2}{|x-y|^2} \, dx \, dy \\ 
		& \leq \frac{1}{h N} \int_0^l \int_0^l  \frac{| u_l(x)  - u_l(y)  |^2}{|x-y|^2} \, dx \, dy = \frac{l }{h N} \, c_l= \frac{l }{l - r(h,l)} \, c_l
		\end{aligned}
		\end{equation}
		Set $\underline{c} := \liminf_{l \to +\infty} c_l$ and $\overline{c} := \limsup_{l \to +\infty} c_l$,  and recall that  $\overline{c} < + \infty$. 
		
		Let $ \{l_n\}, \, \{L_n\}$  be such that $c_{l_n} \geq \overline{c}- \frac 1n$,  $c_{L_n} \leq \underline{c}+ \frac 1n$ for every $n\in\N$ 
		and $\frac{L_n}{l_n} \to +\infty$ as $n\to +\infty$ . By \eqref{thm:cl:1} with $h$ replaced by $l_n$ and $l$ replaced by $L_n$, we have
		\[
		\overline{c}- \frac 1n  \leq c_{l_n} \leq \frac{L_n}{L_n - r(l_n,L_n)} \, c_{L_n} \leq \frac{L_n}{L_n - r(l_n,L_n)}  \, (\underline{c} + \frac 1n) \,.  
		\]   
		Since $r(l_n,L_n)$ is bounded by $l_n$ and recalling that $\frac{L_n}{l_n} \to +\infty$, by taking the limit as $n \to +  \infty$ in the above inequalities we obtain
		$\overline{c}  \leq \underline{c}$, and we clearly deduce that equality holds and denote by $c_\infty:=\overline{c}=\underline{c}$ such a quantity. 
		
Finally, we show that $c_\infty$ is positive: By the condition on the derivatives of maps belonging to $\mathcal{U}_{1}$, it is immediate to see that $c_1>0$. By setting   $h = 1$ in  \eqref{thm:cl:1} and letting $l\to +\infty$ we then infer that $c_\infty>0$.  
\end{proof}

We can now define the candidate $\Gamma$-limit for $F^l$ as 
\begin{equation} \label{def:Flim} 
 F^\infty(w):=\begin{cases} \displaystyle
               \int_0^1 \int_{0}^1 \frac{|w(x)-w(y)|^2}{|x-y|^2} + c_\infty &\text{if } w\in H^{\frac 12}(0,1), \\
               +\infty&\text{otherwise,}
              \end{cases}
\end{equation}
where the constant $c_\infty$ is defined in   Theorem \ref{thm:cl}.

As a consequence of Theorem \ref{thm:cl} we also obtain the following corollary:

\begin{corollary}\label{cor1}
For any $l>0$ let $x=x_l\in [0,l]$ be an arbitrary point. Let $u_l \in \mathcal{U}_l$ be a minimizer for $c_l$. Then 
\begin{align}
\lim_{l\to+\infty}\frac{1}{l}\int_{0}^{x_l}\int_{x_l}^l\frac{| u_l(x)  - u_l(y)  |^2}{|x-y|^2} \, dx \, dy=0. 
\end{align}
\end{corollary}
\begin{proof}
By Theorem \ref{thm:cl} 
the function $s\mapsto R(s):=c_s-c_\infty$ is bounded and such that $R(s)\rightarrow0$ as $s\rightarrow+\infty$.
Therefore, for every sequence $\{x_l\}_l$ we easily infer
$$\frac{x_l}{l}R(x_l)\to 0,\;\;\;\;\;\;\;\frac{l-x_l}{l}R(l-x_l)\to 0,\;\;\;\text{ as } \,\, l\to+\infty.$$
Having said that, we write
\begin{align}
 c_l&=\frac{1}{l}\int_{0}^{x_l}\int_{0}^{x_l}\frac{| u_l(x)  - u_l(y)  |^2}{|x-y|^2} \, dx \, dy+\frac{1}{l}\int_{x_l}^{l}\int_{x_l}^{l}\frac{| u_l(x)  - u_l(y)  |^2}{|x-y|^2} \, dx \, dy\nonumber\\
 &\;\;\;+ \frac{2}{l}\int_{0}^{x_l}\int_{x_l}^l\frac{| u_l(x)  - u_l(y)  |^2}{|x-y|^2} \, dx \, dy\nonumber\\
 &\geq\frac{x_l}{l}c_{x_l}+\frac{l-x_l}{l}c_{l-x_l}+\frac{2}{l}\int_{0}^{x_l}\int_{x_l}^l\frac{| u_l(x)  - u_l(y)  |^2}{|x-y|^2} \, dx \, dy\nonumber\\
 &= c_\infty+\frac{x_l}{l}R(x_l)+\frac{l-x_l}{l}R(l-x_l)+\frac{2}{l}\int_{0}^{x_l}\int_{x_l}^l\frac{| u_l(x)  - u_l(y)  |^2}{|x-y|^2} \, dx \, dy,
\end{align}
and taking the limsup as $l\to+\infty$ entails
$$c_\infty \geq c_\infty+\limsup_{l\to+\infty}\frac{2}{l}\int_{0}^{x_l}\int_{x_l}^l\frac{| u_l(x)  - u_l(y)  |^2}{|x-y|^2} \, dx \, dy\geq c_\infty \, ,
$$
from which the thesis follows.
\end{proof}

\begin{proposition} \label{rmk:cl}
Let $u_l \in \mathcal{U}_l$ be a minimizer of problem \eqref{def:cl} for each $l>0$, and let $w_l:= w_{u_l}$ be defined as in \eqref{wl} and $a_l:= \int_0^1 w_l \, dx$. Then  
\[
(w_l - a_l) \weak 0 \,\, \text{ weakly in } \,\, H^{\frac 12}(0,1) \,\, \text{ as } \,\, l \to +\infty \,.
\]
\end{proposition}

\begin{proof}
Set $\tilde{w}_l:=w_l - a_l$. 
By Theorem \ref{thm:cl} one can easily check that $\tilde{w}_l$ has bounded $\dot H^{\frac 12} (0,1)$-seminorm;
therefore, by Poincar\'e-Wirtinger inequality, up to a subsequence, $\tilde w_l \weak \tilde w$ in $H^{\frac 12}(0,1)$.  
Moreover, by the change of variables $x'=lx$, $y'=ly$ and by Corollary \ref{cor1} applied to $x_l=l/2$ we infer
\[
\int_0^{\frac12} \int_{\frac12}^{1} \frac{|\tilde{w}_l(x)-\tilde{w}_l(y)|^2}{ |x-y|^2 } \, dx \, dy  = 
 \frac{1}{l} \int_0^{\frac{l}{2}} \int_{\frac{l}{2}}^{l} \frac{|u_l(x)-u_l(y)|^2}{ |x-y|^2 } \, dx \, dy \to 0,
\]	
as $l \to +\infty$. By lower semi-continuity we deduce that
$$
\int_0^{\frac12} \int_{\frac12}^{1} \frac{|\tilde{w}(x)-\tilde{w}(y)|^2}{ |x-y|^2 } \, dx \, dy  = 0,
$$
which clearly implies that $\tilde w$ is constant, and in fact, since it has zero average, $\tilde w \equiv 0$. We conclude that the whole sequence $w_l - a_l $
weakly converges to $0$ in $H^{\frac 12}(0,1)$.
\end{proof}

Finally, the fact that minimizers weakly converge to zero leads to the fact that optimal dislocation configurations tend to be uniformly distributed in the limit
as $l\to +\infty$.

\begin{theorem}\label{2.6}
Let $w_l$ be a family of minimizers of $F^l$, and let 
$X_l=\{x_1,\ldots,\ldots x_{N_l}\}$  be the corresponding family of configurations of dislocations defined as in \eqref{diw}. 
Then, setting $\mu_l:= \frac{1}{l}\sum_{i=1}^{N_l} \delta_{\frac{x_i}{l}}$ we have that $\mu_l\weakstar \frac{\Lambda}{\delta(\Lambda + \lambda)}$ as $l\to +\infty$. 
\end{theorem}

\begin{proof}
Let us fix $0<p<q<1$; by Proposition \ref{rmk:cl} we have that, up to an additive constant, $w_l\to 0$ in $L^1$. Therefore, there exist $p_l \to p$, $q_l\to q$ such that
$r_l:=w_l (q_l) - w_l(p_l) \to 0$ as $l\to +\infty$. Then, setting $M_l:= \sharp \frac{X_l}{l}\cap (p_l,q_l)$ we have
\[
r_l = w_l (q_l) - w_l(p_l)  = \int_{p_l}^{q_l} w_l' \, dx
= -\frac{M_l}{\sqrt l} (\Lambda + \lambda) \delta + \lambda \sqrt{l} (q_l - p_l) + \tilde r_l \,,
\]
where $\tilde r_l \to 0$ as $l\to +\infty$.
We deduce that
$$
\frac{M_l}{l}= \frac{\lambda}{\delta(\Lambda + \lambda)} (q-p) + \hat r_l\,,
$$
where $\hat r_l\to 0$ as $l\to +\infty$. By the arbitrariness of $p$ and $q$ in $(0,1)$, we conclude the thesis. 
\end{proof}

\section{$\Gamma$-convergence}

\begin{theorem}[$\Gamma$-convergence]\label{mainthm}
As $l \to +\infty$, the functionals $F^l$ defined in \eqref{defl} $\Gamma$-converge with respect to the weak topology of $H^{\frac 12}(0,1)$ to the functional $F^\infty$, defined in \eqref{def:Flim}.

\end{theorem}

\subsection{Proof of $\Gamma$-liminf inequality}
	Let $w_l \in \mathcal W_l $ be such that $w_l \weak w$ weakly in $H^{\frac12}$.
	Let $l\mapsto M_l\in(0,l) \cap \N$ be  such that $M_l \to +\infty$ as $l \to +\infty$ and
	\begin{equation} \label{gamm:1}
	 \lim_{l\to +\infty} \frac{M_l}{l} =0 \,.
	\end{equation}
	Let $x_i := \frac{i}{M_l}$ for $i=0,\dots,M_l$  and set $I_i:=(x_{i-1},x_i)$ for $i=1,\dots,M_l$. 
	In order to show the $\Gamma$-liminf inequality, we decompose the energy as
	\[
	\begin{aligned}
	\norh{w_l}^2 & = \sum_{i \neq j}^{M_l} \int_{I_i} \int_{I_j} \frac{| w_l(x) - w_l (y)  |^2}{|x-y|^2} \, dx \, dy + \sum_{i =1}^{M_l} \int_{I_i} \int_{I_i} \frac{| w_l(x) - w_l (y)  |^2}{|x-y|^2} \, dx \, dy \\
	  & =: J_l^1 + J_l^2 \,.  
	\end{aligned}
	\]
	Let us first estimate $J^2_l$. For $x \in (0,l/M_l)$ and each $i=1,\dots,M_l$, define the rescaled function
	\[
	u_l^i(x):= \sqrt{l}  w_l \left( \frac{x}{l} + x_{i-1} \right) \,.
	\]
	By its very definition $u_l^i \in \mathcal U_{\frac{l}{M_l}}$, and whence it  is a competitor for $c_{\frac{l}{M_l}}$, as defined in \eqref{def:cl}. Therefore, by introducing the new variables $x=x'/l + x_{i-1}$, $y=y'/l + x_{i-1}$ and recalling that $x_i = x_{i-1}+1/M_l$, we infer 
	\[
	\begin{aligned}
	 \int_{I_i} \int_{I_i} & \frac{ |w_l (x)  - w_l (y) |^2}{|x-y|^2} \, dx \, dy =\\  
	 & =  \int_0^{l/M_l} \int_0^{l/M_l} \frac{|w_l (\frac{x'}{l} + x_{i-1}) - w_l (\frac{y'}{l} + x_{i-1})|^2}{|x'-y'|^2} \, dx' \, dy' \\
	 & =  \frac{1}{l} \int_0^{l/M_l} \int_0^{l/M_l} \frac{|u_l^i (x') - u_l^i (y')|^2}{|x'-y'|^2} \, dx' \, dy' \geq \frac{1}{M_l} \, c_{\frac{l}{M_l}} \,. 
	\end{aligned}
	\]
	Hence
	\[
	J^2_l \geq c_{\frac{l}{M_l}} \,.
	\]
	By assumption \eqref{gamm:1} we have $l/M_l \to +\infty$ as $l \to +\infty$, therefore we can apply Theorem \ref{thm:cl} and conclude that 
	\[
	\liminf_{l \to +\infty} J^2_l \geq \lim_{l \to +\infty} c_{\frac{l}{M_l}} = c_\infty \,.
	\]
	From the above inequality we get
	\begin{equation} \label{gamm:3}
	\liminf_{l \to +\infty} \norh{w_l}^2 = \liminf_{l \to +\infty} (J^1_l + J^2_l) \geq \liminf_{l \to +\infty} J^1_l   + c_\infty \,.
	\end{equation}
	We are left to estimate $J^1_l$, and we have to prove that 
	\begin{equation} \label{gamm:5}
	\liminf_{l \to +\infty} J^1_l \geq \norh{w}^2 \,.
	\end{equation}
	Let $Q:=[0,1] \times [0,1]$, $D_l:=\cup_{i=1}^{M_l} I_i \times I_i$ and $Q_l := Q \smallsetminus D_l$.  For a.e.~$(x,y) \in Q$ define
	\[
	g (x,y) := \frac{|w (x) - w(y)|}{|x-y|}  \,, \quad g_l (x,y) := \frac{|w_l (x) - w_l(y)|}{|x-y|} \,, \quad \tilde{g}_l := \rchi_{Q_l} g_l\,.
	\]
	Now notice that, as $l \to +\infty$,
	\begin{equation} \label{gamm:6}
	g_l \weak g \quad \text{ weakly in } \quad L^2(Q) \,.
	\end{equation}
	Indeed, since we are assuming that $w_l \weak w$ weakly in $H^{\frac 12}(0,1)$, by compact Sobolev embedding, one has that $(w_l - \int_0^1 w_l \, dx )\to (w - \int_0^1 w \, dx)$ strongly in $L^2(0,1)$ (by  a Poincar\'e-Wirtinger type inequality we can control the full norm of $(w_l - \int_0^1 w_l \, dx )$ in $H^{\frac12}$). Therefore, up to subsequences, we also have $(w_l - \int_0^1 w_l \, dx )\to (w - \int_0^1 w \, dx)$ a.e.~in $(0,1)$. By definition we then have $g_l \to g$ a.e. in $Q$. Since $\nor{g_l}_{L^2(Q)}= \norh{w_l}$ is uniformly bounded, we also have (along the subsequence) $g_l \weak g$ weakly in $L^2(Q)$. Since the limit does not depend on the subsequence, we conclude \eqref{gamm:6}. 
	Observe that $\rchi_{Q_l} \to 1$ strongly in $L^p(Q)$ for every $1\leq p < \infty$, since $|D_l|=  1/M_l \to 0$ as $l \to +\infty$. Therefore from \eqref{gamm:6} we conclude that $\tilde{g}_l \weak g$ weakly in $L^2(Q)$, and so \eqref{gamm:5} follows by lower semicontinuity, upon noticing that $\nor{\tilde{g}_l}^2_{L^2(Q)}= J_l^1 $.
From \eqref{gamm:3}, \eqref{gamm:5} we conclude the $\Gamma$-liminf inequality.

\subsection{Proof of the $\Gamma$-limsup inequality}
In order to prove the $\Gamma$-limsup inequality, we need the following Lemma.
\begin{lemma}\label{bpd}
Let $u_l$ be a minimizer for  \eqref{def:cl}. Let moreover $A_l\subset [0,l]$ be an open set such that $|A_l | /l \to 0$ as $l\to +\infty$, and $A_l$ is 
union of intervals whose length is larger than some constant $C>0$ independent of $l$. Then
$$
\lim_{l\to +\infty} \frac{1}{l}  \int_{A_l}\int_0^l \frac{|u_l(x) - u_l(y)|^2}{|x-y|^2} \, dx \, dy = 0. 
$$  
\end{lemma}
\begin{proof}
Let $l_n\to +\infty$ be such that
$$
\lim_{n\to +\infty} \frac{1}{l_n}  \int_{A_{l_n}}\int_0^{l_n} \frac{|u_{l_n}(x) - u_{l_n}(y)|^2}{|x-y|^2} \, dx \, dy =   \limsup_{l\to +\infty} \frac{1}{l}  \int_{A_l}\int_0^l \frac{|u_l(x) - u_l(y)|^2}{|x-y|^2} \, dx \, dy =: E. 
$$
By the optimality of $u_l$ we have
\begin{equation}\label{poe}
\lim_{n\to +\infty} \frac{1}{l_n}  \int_{A^c_{l_n}}\int_0^{l_n} \frac{|u_{l_n}(x) - u_{l_n}(y)|^2}{|x-y|^2} \, dx \, dy = c_\infty - E \,,
\end{equation} 
where we denote by $A_{l_n}^c$ the complement of $A_{l_n}$ in $[0,l]$.
Fix $N\in\N$, and let $\{B^N_{i,n}\}_i$ be the connected components of $A_{l_n}^c$ whose length is at least $N$. 
By the assumption $|A_l|/l \to 0$ as $l \to +\infty$, we infer
$$
\lim_{n\to +\infty}  \frac{ \sum_i |B^N_{i,n}|}{l_n} =1\,.
$$
Hence 
by \eqref{poe}, there exists at least one element of $\{B^N_{i,n}\}_i$, which we name $(x_{n,N},y_{n,N})$,  such that 
$$
\limsup_{n \to +\infty} \frac{1}{y_{n,N} - x_{n,N}}  \int_{x_{n,N}}^{y_{n,N}} \int_{x_{n,N}}^{y_{n,N}} \frac{|u_{l_n} (x) - u_{l_n} (y)|^2}{|x-y|^2} \, dx \, dy \le c_\infty -E. 
$$
By a diagonal argument, there exists a subsequence $\tilde l_N = l_{n_N}$  and intervals $(x_N,y_{N})$ such that 
$$
\lim_{N \to + \infty} \frac{1}{y_N - x_N}  \int_{x_N}^{y_N} \int_{x_N}^{y_N} \frac{|u_{\tilde l_N} (x) - u_{\tilde l_N} (y)|^2}{|x-y|^2} \, dx \, dy \le c_\infty -E. 
$$
Setting $\tilde u_N(x):=  u_{\tilde l_N} (x - x_N)$ we have
$$
\lim_{N \to + \infty} \frac{1}{y_N - x_N}  \int_0^{y_N - {x_N}} \int_0^{y_N - {x_N}}  \frac{|\tilde u_{N} (x) - \tilde u_{N} (y)|^2}{|x-y|^2} \, dx \, dy \le c_\infty -E. 
$$
Since $\tilde u_N \in\mathcal U_{y_N - x_N}$ and recalling that $y_N - x_N\to +\infty$, by Theorem \ref{thm:cl} we conclude that $E=0$, ending the proof.  
\end{proof}

Now, let $w \in H^{\frac 12}(0,1)$. We want to construct a recovery sequence $g_l \in \mathcal{W}_l$ such that 
\begin{gather}
g_l \weak w   \quad \text{weakly in} \quad H^{\frac 12}(0,1) 
\label{sup:1} \quad \text{as} \quad l \to +\infty \,, \\   
\limsup_{l \to +\infty}F^l(g_l) = F^{\infty} (w)\,, \label{sup:2}
\end{gather}
where $F^l$ is defined in \eqref{defl} and $F^{\infty}$ in \eqref{def:Flim}. 

By standard density arguments in $\Gamma$-convergence we may assume that $w$ is piece-wise affine. Specifically, without loss of generality we  assume that $w\in C^0(0,1)$ is of the form $w=\sum_{i=1}^m (c_i+ \alpha_i x) \chi_{I_i}$ where $c_i \in\R, \, \alpha_i \in \R \setminus \{0\}$, and $\{I_i\}$ is a partition of $(0,1)$. 

Let $\{w_l\}\subset W^{1,\infty}(0,1)$ be a family of minimizers provided by Proposition \ref{rmk:cl} and with average equal to zero,  and let  $X_{w_l}$ be a set of dislocations associated to $w_l$, defined as in \eqref{diw}. 
Now we want to plug extra dislocations $N_{w_l}$, inducing the macroscopic strain $w'$. In principle, we need periodically distributed dislocations whose density depends on $\alpha_i$. However, some care is needed to ensure that the new dislocations are not plugged on top of the dislocations already present, namely, they should be introduced in such a way that $X_{w_l} \cup N_{w_l}$ gives back an admissible configuration of dislocations. To this purpose, it is easy to see that there exists a finite family of points $N_{w_l}=\{x_1, \ldots x_{N_l}\} \subset (0,1)$ with  $x_i<x_{i+1}$ for all $i$ with the following properties:

\begin{itemize}
\item[i)]  $|x-z| \ge \frac{\delta}{l}$ for all  $x\in N_{w_l}$, $z\in \frac 1l X_{w_l}$;  
\item[ii)]  The distance between any pair of consecutive points is prescribed up to errors of order $\frac{\delta}{l}$ as follows:
\begin{align*} 
& \left||x_{i+1}-x_i| - \frac{\Lambda \delta}{- \alpha_j \sqrt{l}}\right| \le \frac{\delta}{l} \qquad &\text{ if } x_i,\,x_{i+1} \in N_{w_l}\cap I_j, \,  \alpha_j <0,
\\
& \left||x_{i+1}-x_i| - \frac{\lambda \delta}{\alpha_j \sqrt{l}} \right| \le \frac{\delta}{l} \qquad &\text{ if } x_i,\,x_{i+1} \in N_{w_l}\cap I_j, \,  \alpha_j >0.  
\end{align*} 
\item[iii)]  There exists $C>0$ such that, for every open interval $G\subset (0,1)$ with $|G|\ge \frac{C}{\sqrt l}$ we have that, for $l$ large enough,    $N_{w_l}\cap G \neq \emptyset$.
\end{itemize}
In other words, property ii) enstablishes that the distance between points in $N_{w_l}$ is prescribed, of order $\frac{1}{\sqrt{l}}$, and depends on the derivative $\alpha_j$ of $w$ on $I_j$. The errors of order $\frac{\delta}{l}$ are admitted in order to guarantee i). The last condition instead ensures that we cover any interval $I_j$ with points in $N_{w_l}$ without creating holes of order larger than $\frac{C}{\sqrt{l}}$ between two consecutive intervals.

Let $\f:(0,1)\to (0,1+ \frac{\delta N_l}{l} )$ be defined by $\f(x) := x + \frac{\delta}{l}\sharp \{N_{w_l} \cap (0,x)\}$, and, with a little abuse of notation, let $\f^{-1}: (0,1+ \frac{\delta N_l}{l} ) \to (0,1)$     
be the function that coincides with the inverse of $\f$ on its image, and extended to the whole interval $ (0,1+ \frac{\delta N_l}{l})$ so that it is continuous and monotone
(such an extension clearly exists and is unique). 

Now, we let $P(t):= (1-(\f^{-1})' (t))$ for all $t\in (0,1+ \frac{\delta N_l}{l} )$, and set $\uplambda: (0,1) \to \R$ 
to be equal either to $\lambda$ or to $\Lambda$, if $w'(x)$ is positive or negative, respectively. Then, we set 
$$
\tilde w_l(x):=w_l( \f^{-1}(x)), \quad
\tilde g_l (x) := 
 w(0) +
\int_0^x P (t) \sqrt{l}\uplambda(\f^{-1}(t)) \, dt,  \quad \text{ for all } x\in (0,1+ \frac{\delta N_l}{l} ).
$$

Now we are in a position to introduce the recovery sequence 
\begin{equation}\label{defrecs}
g_l:= (\tilde w_l + \tilde g_l)\res [0,1].  
\end{equation}
By construction $g_l$ is admissible; the check is left to the reader. 
\smallskip

First, we show that $\tilde g_l \res (0,1) \to w$ strongly in $L^\infty(0,1)$. More precisely, we shall prove that
\begin{equation}\label{stili}
\|\tilde g_l -  w \|_{L^\infty (0,1)} \le \frac{C}{\sqrt l}
\end{equation}
for some $C\in \R$. 
To this purpose, it is enough to estimate $\tilde g_l(x) - w(x)$ only for $x\in I_1$, since such an estimate can be clearly iterated for the remaining (finite) intervals. 
Without loss of generality, we assume $\alpha_1 >0$. 
By properties  ii) and iii) above we have that, for all $x\in I_1$,
$$
\Big|\sharp \{ N_{w_l} \cap (0,x) \} - x \frac{\alpha_1 \sqrt{l}}{ \lambda \delta} \Big| \le C,
$$
for some constant $C$ independent of $l$. As a consequence, by a change of variables and by its very definition,
\begin{equation}\label{stili2}
\Big |\tilde g_l(\f(x)) - w(0) -  \sharp \{ N_{w_l} \cap (0,x) \} \frac{\delta \lambda}{\sqrt{l}}  \Big |  \le \frac{\delta \lambda}{\sqrt l} \qquad \text{ for all } x\in I_1.  
\end{equation}
Setting $I_1 = [0,p_1]$, since
\begin{equation}\label{f-id}
 |\f(p_1) - p_1| \le \frac{C}{\sqrt{l}},
\end{equation}
we deduce that
\begin{align}\label{stili3}
\Big |w(\f(x)) - w(0) -  \sharp \{ N_{w_l} \cap (0,x) \} \frac{\delta \lambda}{\sqrt{l}}  \Big |   
\le
 |\alpha_1 ( \f(x) -x) | +  \frac{C}{\sqrt l} \le   \frac{C}{\sqrt l} 
 \end{align}
holds for all $x\in \f^{-1} (I_1)\subset I_1$ and $l$ large enough. Moreover, thanks to \eqref{f-id}, we conclude that \eqref{stili3}, in fact, holds true on the whole $I_1$.
This, together with \eqref{stili2} and by triangular inequality yields,
$$
|\tilde g_l(\f(x)) -  w(\f(x)) | \le \frac{C}{\sqrt l} \qquad \text{ for all } x\in I_1,
$$
from which  \eqref{stili} easily follows. 

Now we prove that
\begin{equation}\label{stilh}
\|\tilde g_l -  w \|_{H^{\frac12} (0,1)} \to 0 \quad \text{as} \quad l\to +\infty.
\end{equation}
Given $M\in\N$, we set $J_{\frac Ml}(x) := (x-\frac  Ml, x+ \frac  Ml) \cap[0,1]$ for every $x\in (0,1)$ and define $h_l:= \tilde g_l - w$.  Moreover, we set 
\begin{equation} \label{deffi}
F_l:= \{x\in [0,1] :  \tilde g_l' \equiv 0 \text{ on } J_{\frac Ml}(x) \}\,.
\end{equation}
Note that $|F_l^c|\leq \frac{CM}{\sqrt{l}}$, since the number of points in $N_{w(l)}$ is of order $\sqrt{l}$.
Recalling \eqref{stili}:
\begin{align*}
\|\tilde g_l -  w \|_{\dot H^{\frac12} (0,1)} 
& =
\int_0^1 \Big( \int_{J_{\frac Ml}(x)} \frac{|h_l(x) -h_l(y)|^2}{|x-y|^2} \, dy + \int_{J^c_{\frac Ml}(x)} \frac{|h_l(x) -h_l(y)|^2}{|x-y|^2} \, dy \Big) \, dx
\\
& \le 
\int_{F_l}  \int_{J_{\frac Ml}(x)} C \, dy \, dx+ 
\int_{F^c_l}  \int_{J_{\frac Ml}(x)} Cl \, dy \, dx+ 
\int_0^1 \frac Cl \int_{J^c_{\frac Ml}(x)} \frac{1}{|x-y|^2} \, dy  \, dx
\\
& \le
\frac{CM}{l} + |F^c_l|  \frac{2M}{l}  Cl + \int_0^1  \frac{C }{M}  \, dx
\\
& \le
\frac{CM}{l} +  \frac{C M^2}{\sqrt{l}}   + \int_0^1  \frac{C }{M}  \, dx,
\end{align*}
where in the first inequality we have used that $w'$ is of order $1$,  that $\tilde g_l'$ is of order $\sqrt{l}$ on $F^c_l$, and that $h_l$ is of order $\frac{1}{\sqrt{l}}$. Notice that the last term converges to  $\frac CM$ as $l\to +\infty$. By sending $M\to +\infty$ and recalling \eqref{stili} we hence deduce \eqref{stilh}.

It remains to show that
\begin{equation}\label{boundene}
\limsup_{l \to +\infty} \|\tilde w_l \|^2_{\dot H^{\frac12} (0,1)} \le c.
\end{equation}
Indeed,  from \eqref{boundene} it follows that $\tilde w_l$, up to translations, is pre-compact in $H^{\frac 12}$, and by Proposition \ref{rmk:cl}  it converges, still up to translations,  to zero  in measure. Recalling that $\tilde w_l$ have zero mean, we deduce  that
$\tilde w_l\weak 0$ in $ H^{\frac 12}(0,1)$, that together with \eqref{stilh} and \eqref{defrecs} yields \eqref{sup:1}. Moreover, by \eqref{defrecs},  \eqref{stilh}, and the definition of $c_\infty$,  we deduce \eqref{sup:2} as follows
\begin{align*}
\limsup_{l \to +\infty} F^l(g_l) & = \limsup_{l \to +\infty} \Big[ \|\tilde w_l \|_{\dot H^{\frac 12} (0,1)}^2 + \|\tilde g_l \|_{\dot H^{\frac 12} (0,1)}^2 +
2 \langle \tilde w_l , \tilde g_l \rangle_{\dot H^{\frac 12}}\Big] 
\\
& \le  \limsup_{l \to +\infty} \|\tilde w_l \|_{\dot H^{\frac 12} (0,1)}^2 + \limsup_{l \to +\infty} \|\tilde g_l \|_{\dot H^{\frac 12} (0,1)}^2\\
& \le c_\infty + \| w \|_{\dot H^{\frac 12} (0,1)}^2 = F^\infty(w).
\end{align*}

We will now prove \eqref{boundene}. Since $|\f(s) - \f(t)|\ge |s-t|$ for all $s, \, t\in (0,1)$, we have   
\begin{align}
\nonumber
\|\tilde w_l \|^2_{\dot H^{\frac12} ([0,1 + \frac{\delta N_l}{l}])} 
& \le \int_0^1 \int_0^1 \frac{ ( w_l(s) - w_l(t) )^2}{| s - t|^2}  ds \,  dt 
\\
\label{seca}
& + \sum_{\overset{x_i \neq x_j}{x_i,\, x_j \in N_{w_l}}} 
\frac{\delta^2}{ l^2} \frac{ ( w_l(x_i) - w_l(x_j) )^2}{| x_i - x_j|^2}
+  \frac{2 \delta}{ l} \sum_{x_i\in N_{w_l}}  \int_0^1 \frac{ ( w_l(x_i) - w_l(t) )^2}{| x_i - t|^2} \,   dt.
\end{align}
The first term is uniformly bounded by a constant independent of $l$; Therefore, we have to prove that the terms in \eqref{seca} tend to $0$ as $l\to +\infty$. 
To this purpose, for all $x_i\in N_{w_l}$ we set $I^l_i:= (x_i - \frac{\delta}{2l}, x_i + \frac{\delta}{2l})$, and we denote by $U^l$ their union. 
We have %
\begin{align}
\nonumber
\sum_{\overset{x_i \neq x_j}{x_i,\, x_j \in N_{w_l}}} 
\frac{\delta^2}{ l^2} \frac{ ( w_l(x_i) - w_l(x_j) )^2}{| x_i -  x_j|^2}
& = 
\sum_{\overset{x_i \neq x_j}{x_i,\, x_j \in N_{w_l}}}
\int_{I^l_i} \int_{I^l_j}  \frac{ ( w_l(x_i) - w_l(x_j) )^2}{| x_i -  x_j|^2} \, ds \, dt
\\
\nonumber
& \le 
2 \sum_{\overset{x_i \neq x_j}{x_i,\, x_j \in N_{w_l}}}
\int_{I^l_i} \int_{I^l_j}  \frac{ ( w_l(x_i) - w_l(x_j) )^2}{|s-t|^2} \, ds \, dt
\\
\label{terza}
& \le
C \int_{U^l} \int_{U^l}  \frac{ ( w_l(s) - w_l(t) )^2}{|s-t|^2} \, ds \, dt
\\
\label{quarta}
& + C \sum_{x_i,\, x_j \in N_{w_l}}    \int_{I^l_i} \int_{I^l_j}  \frac{ ( w_l(x_i) - w_l(t))^2}{|s-t|^2} \, ds \, dt 
\end{align}

The term in \eqref{terza} converges to zero, thanks to Lemma \ref{bpd} and via a change of variables. The term in \eqref{quarta} can be easily estimated first integrating in $s$,  exploiting the fact that $w_l'\leq \Lambda\sqrt{l}$, and property ii), so that 
$$
C \sum_{x_i,\, x_j \in N_{w_l}}    \int_{I^l_i} \int_{I^l_j}  \frac{ ( w_l(x_i) - w_l(t))^2}{|s-t|^2} \, ds \, dt  \le 
C (\sharp N_{w_l})^2 \frac{\delta^2}{l^2} \Big(\Lambda \sqrt{l} \frac{\delta}{l} \Big ) ^2  l \le \frac{C}{l},
$$  
which also converges to $0$ as $l\to +\infty$. 
It remains to estimate the second term in \eqref{seca}; we have

\begin{equation}\label{stilu}
\begin{aligned}
\frac{2 \delta}{ l} \sum_{x_i\in N_{w_l}}  \int_0^1 \frac{ ( w_l(x_i) - w_l(t) )^2}{|x_i - t |^2} \,   dt
=
2 \sum_{x_i\in N_{w_l}}  \int_0^1 \int_{I_i} \frac{ ( w_l(x_i) - w_l(t) )^2}{|x_i - t |^2} \, ds \,    dt
\\
= 2 \sum_{x_i\in N_{w_l}}  \int_{I_i} \int_{I_i} \frac{ ( w_l(x_i) - w_l(t) )^2}{|x_i - t |^2} \, ds \,    dt
+
2 \sum_{x_i\in N_{w_l}}  \int_{I_i^c} \int_{I_i} \frac{ ( w_l(x_i) - w_l(t) )^2}{|x_i - t |^2} \, ds \,    dt
\\
= 2 \sum_{x_i\in N_{w_l}}  \int_{I_i} \int_{I_i} \frac{ ( w_l(x_i) - w_l(t) )^2}{| x_i - t|^2} \, ds \,    dt
\\
+
C \sum_{x_i\in N_{w_l}} \Big[ \int_{I_i^c} \int_{I_i} \frac{ ( w_l(x_i) - w_l(s) )^2}{| x_i -  t |^2} \, ds \,    dt
+
\int_{I_i^c} \int_{I_i} \frac{ ( w_l(s) - w_l(t) )^2}{|s - t|^2} \, ds \,    dt \Big] .
\end{aligned}
\end{equation}  
The last term tends to zero, again thanks to Lemma \ref{bpd} and via a change of variables. Moreover, since $w_l'$ is bounded by $C\sqrt{l}$, 
$$
2 \sum_{x_i\in N_{w_l}}  \int_{I_i} \int_{I_i} \frac{ ( w_l(x_i) - w_l(t) )^2}{| x_i - t|^2} \, ds \,    dt \le C \sqrt{l} \frac{1}{l^2} l,
$$
and the right hand side converges to $0$ as $l\to +\infty$. Finally,  integrating in $t$ and using again the Lipschitz continuity of $w_l$, we get
$$
C \sum_{x_i\in N_{w_l}}  \int_{I_i^c} \int_{I_i} \frac{ ( w_l(x_i) - w_l(s) )^2}{| x_i -  t |^2} \, ds \,    dt \le
C \sum_{x_i\in N_{w_l}}  \int_{I_i}  \frac{\quad \frac{C}{l} \quad}{ \frac{\delta}{2l}} \, ds  \le \frac{C}{\sqrt{l}}. 
$$
Casting these estimates in \eqref{stilu}, we deduce that also the last term in \eqref{seca} tends to zero, which in turn yields \eqref{boundene} and concludes the proof of the $\Gamma$-limsup inequality.
\vskip50pt

\section{Periodicity of dislocations on $\mathcal S^1$}\label{periodicity}

\subsection{Admissible configurations and the energy functional}
In order to study the optimal positioning of dislocations we restrict ourselves to the analysis of a simplified model. 
Roughly speaking, we neglect boundary effects by working on $\mathcal S^1$; then, we will consider the limit of the  energy induced by a finite number of dislocations (on $\mathcal S^1$) as $\Lambda\to +\infty$. 
To this purpose, we  consider the new distance on $(0,1)$ defined by 
\[
d(x,y)=\min \{|x-y|, 1-|x-y|\}.
\]
We fix the number $N\in\N$ of dislocations, 
and  consider families of points $(x_1,\dots,x_N)\in [0,1]^N$ which represent the dislocation positions. For convenience we will ``cut and paste'' the dislocations on the whole $\R$,  by setting
\begin{equation}\label{points_yi}
\{ y_i\}_{i \in I}:= \{y\in \R: y= x_j+ k, \, 1\le j\le N, \, k\in\Z \} \, .
\end{equation}
Assume now 
\begin{equation}\label{deltaN}
\delta=\frac{\lambda}{N(\lambda+\Lambda)}. 
\end{equation}
The class of admissible displacements is defined as
\begin{equation}
\begin{aligned}
 \mathcal V^\Lambda :=\Big\{& v   \in W^{1,\infty}(\R):\exists \,(x_1,\dots,x_N)\in(0,1)^N:d(x_i,x_j)\geq \delta \,\, \forall \, i\neq j \,, \\
&\;v'=\lambda-(\lambda+\Lambda)\sum_{i \in I}\rchi_{E_i},\;E_i:=\left(y_i-\frac{\delta}{2},y_i+\frac{\delta}{2}\right)\Big\} \,,
\end{aligned}
\end{equation}
where the points $y_i$ are defined in \eqref{points_yi}. 
Notice that by definition the function $v'$ is periodic on $\R$ with period equal to $1$, and that the condition 
$\delta=\frac{\lambda}{N(\lambda+\Lambda)}$  enforces $v(0)=v(1)$;
in this way  $v$ is periodic with period equal to $1$ and continuous on the whole $\R$. The energy of the system is given by
\begin{align}
E(v):=\int_0^1\int_0^1\frac{|v(x)-v(y)|^2}{d(x-y)^2}dxdy,
\end{align}
and can be equivalently expressed as in the following Lemma.

\begin{lemma} \label{lem:change}
	For $v \in  \mathcal V^\Lambda $ we have that
	\begin{equation} \label{lem:change:th}
	E(v) = \int_0^1 \int_{ - \frac12}^{ \frac12} \frac{|h(y+z)-h(y)|^2}{|z|^2} \, dz dy - \lambda^2 \,,
	\end{equation}
	where $h(t):=v(t)-\lambda t$. 
\end{lemma}

\begin{proof}
Since $v$ and $d$ are both 1-periodic, then the energy can be computed as
\[
E(v)= \int_0^1 \int_{y-\frac12}^{y+\frac12} \frac{|v(x)-v(y)|^2}{|x-y|^2} \, dx \, dy,
\]
as $d(x-y)=|x-y|$ on the integration domain. Note that
\begin{equation}\label{integranda}
\frac{|v(x)-v(y)|^2}{|x-y|^2} = \lambda^2 + \frac{|h(x)-h(y)|^2}{|x-y|^2} + 2 \lambda \, 
\frac{h(x)-h(y)}{x-y} \,.
\end{equation}
Now recall that $v$ is $1$-periodic, so that $h(y+1)=h(y)-\lambda$. Therefore, by introducing the new variable $z:=x-y$ we have

\begin{equation} \label{lem:change:2}
\begin{gathered}\int_0^1 \int_{ y- \frac12}^{y+ \frac12} \frac{h(x)-h(y)}{x-y} \, dx \, dy   =
\int_{ - \frac12}^{ \frac12} \frac1z \left( \int_0^1 h(y+z) - h(y) \, dy \right)\, dz \, = 
\\
 \int_{ - \frac12}^{ \frac12} \frac1z \left( \int_0^z h(y+1) - h(y) \, dy \right)\, dz \, =
-\lambda. 
\end{gathered}
\end{equation}

Integrating both sides in \eqref{integranda} and again by the change of variable $z:=x-y$,  in view of 
\eqref{lem:change:2}
we conclude \eqref{lem:change:th}.  
\end{proof}

In view of the above lemma, we introduce the class
$$
\mathcal H^\Lambda:=\{ v - \lambda \, Id, \, v \in \mathcal V^\Lambda \} \, . 
$$

Now we are interested in considering the limit as $\Lambda \to +\infty$ of the proposed model; this, recalling \eqref{deltaN}, corresponds to sending $\delta \to 0$.  
Notice that if $h^\Lambda  \in  \mathcal  H^\Lambda$, 
then, up to a subsequence and up to additive constants,  $h^\Lambda$ converges strongly in $L^1(0,1)$ (and in fact in all $L^p$, $p<\infty$) and pointwise almost everywhere to a step function $h$,  as $\Lambda \to +\infty$.
Therefore (since $h$  is not constant) $\|h \|_{\dot H^{\frac 12}}=+\infty$.
To overcome this problem we cut off the core region around dislocation points.
Specifically, for 
fixed $\rho>\frac\delta2 >0$ and $\Lambda>0$ (large enough) we consider the energy functionals $E^\Lambda_{\rho}: \mathcal H^\Lambda \to [0,+\infty)$ defined by 
\begin{align}
E^\Lambda_{\rho}(h) :=  \int_0^1 \int_{ - \frac12}^{-\rho} \frac{|h(y+z)-h(y)|^2}{|z|^2} \, dz dy +
\int_0^1 \int_{ \rho}^{ \frac12} \frac{|h(y+z)-h(y)|^2}{|z|^2} \, dz dy.
\end{align}
Then we study the convergence of the functionals $E^\Lambda_\rho$ as $\Lambda\rightarrow +\infty$.

Let us set $I_\rho:=(-\frac12,-\rho)\cup(\rho,\frac12)$. The functionals above read
\begin{align*}
E^\Lambda_{\rho}(h) 
=\|\Delta_h\|_{L^2([0,1] \times I_\rho)}^2 , 
\end{align*}
where $\Delta_h(y,z)=\frac{h(y+z)-h(y)}{z}$. 
Notice that $h(0) - h(1) = \lambda$ and that the slope of $h$ on $[0,1]\times I_\rho$ is less or equal to $\frac{h(y+z)-h(y)}{\rho}\leq \frac{\lambda}{\rho}$. Hence up to adding  a suitable constant to $h$ we have

\begin{equation}\label{BVbound}
\|h\|_{BV(0,1)} \le  2 \lambda, \qquad    \|\Delta_h(y,z)\|_{L^\infty([0,1] \times I_\rho)} \leq  \frac{\lambda}{\rho} \qquad 
\text{ for all } h\in \mathcal H^\Lambda.
\end{equation}

Let now $ h^\Lambda \in \mathcal H^\Lambda$; up to subsequences,  $h^\Lambda\rightarrow h$ strongly in $L^p(0,1)$, for all $p<\infty$. Without loss of generality we may assume $h^\Lambda\rightarrow h$ a.e. so that $\Delta_{h^\Lambda}\rightarrow \Delta_h$ a.e. in $[0,1]\times I_\rho$. Thanks to the boundedness \eqref{BVbound} we infer $\Delta_{h^\Lambda}\rightarrow \Delta_h$ strongly in $L^2$ and we conclude
\begin{align}\label{Erho}
E^\Lambda_\rho(h^\Lambda)\rightarrow E_\rho(h):=\int \int_{[0,1]\times I_\rho} \frac{|h(y+z)-h(y)|^2}{|z|^2} \, dz dy.
\end{align}

Therefore, the asymptotic behavior of minimizers of   $E^\Lambda_\rho$, as $\Lambda \to +\infty$, is described by the ground states of the  
more tractable functionals $E_\rho$ defined on step functions.

In order to  study the periodicity of minimizers of the energy introduced above, it is convenient to  
rewrite  the energy as a function of the dislocation points.  
We assume that the dislocations are at a minimal distance $\rho$ with $\frac 1N \ge  \rho>0$. Then, we introduce the class of admissible dislocations $\AD^N_{\rho}$ defined as
$$
\AD^N_{\rho}:= \{ \{x_1, \ldots, \,x_N\} \subset [0,1): \, d(x_i,x_j) \ge \rho \quad  \text{ for all } i\neq j\}.
$$
Given $X\in \AD^N_{\rho}$, we set 
$$
Y(X) := \{y_i \in \R: y_i = x_j+ k, \, 1\le j\le N, \,  k\in\Z \} \, .
$$

Now, the energy $E_\rho$ can be regarded as a function of the dislocation points: We introduce the energy functional $\mathcal E^N_{\rho}:
\AD^N_{\rho} \to \R$ defined by 
\begin{equation} \label{kamehameha}
\mathcal E^N_{\rho}(X) = E_\rho(h_X) \quad \text{ for all } X= \{x_1,\dots,x_N\},
\end{equation}
where $h_X \in BV_{loc}(\R)$ is defined, up to an additive constant, by the condition 
$$
h_X'=-\frac{\lambda}{N} \sum_{y\in Y(X)} \delta_{y},
$$
and $E_\rho$ is defined in \eqref{Erho}. 
With a little abuse of notation, given $ X= \{x_1,\dots,x_N\}\in \AD^N_{\rho} $ we will also write 
$\mathcal E^N_{\rho}(x_1,\dots,x_N) = \mathcal E^N_{\rho}(X)$.

The following theorem establishes that the energy $\mathcal E^N_{\rho}$ is minimized on configuration of equi-spaced dislocations.

\begin{theorem} \label{thm:periodicity}
Let $0<\rho<\frac 1N$. The energy $\mathcal E^N_{\rho}$ at \eqref{kamehameha} admits a  minimizer $X\in \AD^N_{\rho}$; moreover each minimizer is of the form $X= \{x_1,\dots,x_N\}$, where  $x_1 < \dots <x_N$ and $d(x_i,x_{i+1}) = \frac{1}{N}$ for every $i = 1, \dots, N-1$. 
\end{theorem}

\subsection{Proof of Theorem \ref{thm:periodicity}}
In order to prove Theorem \ref{thm:periodicity}, we first provide an equivalent formulation for the energy $\mathcal E^N_{\rho}$ at \eqref{kamehameha}, and subsequently exploit its convexity.
Instead of manipulating $\mathcal E^N_{\rho}$ directly, which seems to be an involved and tedious task, we start by computing its first variation in Step 1. After that, we show how the first variation of $\mathcal E^N_{\rho}$ coincides with the one of a new functional $\tilde{\mathcal E}^N_{\rho}$, which is easier to manipulate. Finally, in Step 3, we show that $\tilde{\mathcal E}^N_{\rho}$ is minimized when the points $x_1,\ldots,x_N$ are evenly spaced on $\mathcal{S}^1$: This fact is deduced after noting that $\tilde{\mathcal E}^N_{\rho}$ is a convex function of the $\mathcal{S}^1$-distances between points $x_i$.

\smallskip

\textit{Step 1. Computing the first variation.}
Since $\lambda$ is fixed we assume for simplicity of notation that $\frac{\lambda}{N}=1$. Let us fix a configuration 
$X=\{x_1,\dots,x_N\} \in \AD_\rho^N$, such that 
$d(x_i,x_j) > \rho$ for each $i \neq j$. Fix $i$ and consider the first variation of the energy 
\begin{align}\label{variation}
\lim_{\e\rightarrow 0} \frac{1}{\e}\big(\mathcal E^N_{\rho}(x^\e_1,\dots,x^\e_N)-\mathcal E^N_{\rho}(x_1,\dots,x_N)\big),
\end{align}
where 
\begin{align}
x^\e_j:=\begin{cases}
      x_j&\text{if }j\neq i,\\
      x_{i}+\e&\text{if }j=i.
     \end{cases}
\end{align}
In order to compute the limit in \eqref{variation} we introduce the function $h^\e$ defined, up to additive constants, by  $(h^\e)':=-\sum_i\delta_{x^\e_i}$. 
Let us restrict our analysis to the case $\e>0$, the other case is similar and will yield the same result. Set $h:=h_X$ and notice that
\begin{align}\label{dep}
&h^\e-h = D^\e:=
     \chi_{(x_i,x_i+\e)}.
           \end{align}
Therefore we can write
\begin{align}
\nonumber
\mathcal E^N_{\rho}(x^\e_1,\dots,x^\e_N)-\mathcal E^N_{\rho}(x_1,\dots,x_N)=
&\int_0^1 \int_{I_\rho}\frac{|h^\e(x+z)-h^\e(x)|^2-|h(x+z)-h(x)|^2}{|z|^2}dzdx\\
\label{variation_E}
&=\int_0^1\int_{I_\rho}\frac{(S^\e(x+z)-S^\e(x))(D^\e(x+z)-D^\e(x))}{|z|^2}dzdx,
\end{align}
where we have set
$ S^\e(t)=h(t)+h^\e(t)$, while $D^\e$ is defined in \eqref{dep}.
For $\e$ small enough, we have
\begin{align}
&D^\e(x+z)-D^\e(x)=\begin{cases}
                   -1&\text{if }x\in (x_i,x_i+\e),\\
                   1&\text{if }x+z\in(x_i,x_i+\e),\\
                   0&\text{otherwise.}
                  \end{cases}
     \end{align}

Hence, by \eqref{dep} and \eqref{variation_E} we get
\begin{equation} \label{3l6}
\begin{aligned}
&\mathcal E^N_{\rho}(x^\e_1,\dots,x^\e_N)-\mathcal E^N_{\rho}(x_1,\dots,x_N) \\
&=\int_{I_\rho}\frac{1}{|z|^2}\int_{x_i-z}^{x_i-z+\e}(S^\e(x+z)-S^\e(x))dx-\int_{x_i}^{x_i+\e}(S^\e(x+z)-S^\e(x)) \, dx\,dz \\
&=\int_{I_\rho}\frac{1}{|z|^2}\int_{x_i}^{x_i+\e}(S^\e(x)-S^\e(x-z))-(S^\e(x+z)-S^\e(x)) \, dx \, dz \\
&=2\int_\rho^{\frac12}\frac{1}{|z|^2}\int_{x_i}^{x_i+\e}(S^\e(x)-S^\e(x-z))-(S^\e(x+z)-S^\e(x)) \, dx \, dz \\
&=-2\int_{I_\rho} \frac{1}{|z|^2}\int_{x_i}^{x_i+\e}(S^\e(x+z)-S^\e(x)) \, dx \, dz  \, .
\end{aligned}
\end{equation}
With the aid of an integration by parts, the previous expression equals
\begin{equation}\label{3lines}
\begin{aligned}
&-2\int_{x_i}^{x_i+\e}\frac{S^\e (x+\rho )-S^\e(x)}{\rho}+\frac{S^\e(x-\rho)-S^\e(x)}{\rho}dx\\
&+4\int_{x_i}^{x_i+\e}S^\e \left(x+\frac12 \right)+S^\e \left(x-\frac12 \right)-2S^\e(x)dx\\
&-2\int_{I_\rho}  \int_{x_i}^{x_i+\e}
\frac{\dot S^\e(x+z)}{z} \, dx \, dz \, .
\end{aligned}
\end{equation}
Exploiting the fact that the $\mathcal{S}^1$-distance between the points $x_j$ is larger than $ \rho$, we easily see that, for $\e$ small enough,  $S^\e(x+\rho)-S^\e(x)=-1$  and 
$S^\e(x-\rho)-S^\e(x)=1$ for $x\in (x_i,x_i+\e)$, so that the first line is null.

As for the second line, we will compute it as $\e$ is small; we first see that the values of $S^\e(x-\frac12)$ and $S^\e(x+\frac12)$ do not depend on $\e$ and equal $2h(x-\frac12)$ and $2h(x+\frac12)$, respectively. Moreover $S^\e(x)$ is constant on $(x_i,x_i+\e)$ and coincides with $2h(x)+1=2h^+(x_i)+1$ (where $h^+(t)=\lim_{s\rightarrow t^+}h(s)$), so that
\begin{equation} \label{3l1}  
\begin{aligned}
\lim_{\e\rightarrow0}  \, \frac{1}{\e} & \int_{x_i}^{x_i+\e}S^\e\left(x+\frac12\right)+S^\e\left(x-\frac12\right)-2S^\e(x) \,dx\\
&=2h^+\left(x_i-\frac12 \right)+2h^+\left(x_i+\frac12 \right)-4h^+(x_i)-2.
\end{aligned}
\end{equation}
We write $2h^+(x_i)+1=h^+(x_i)+h^-(x_i)$, so that 
\begin{equation}\label{3l2}
\begin{aligned}
2h^+ \Big(x_i-\frac12 \Big)+2h^+ \Big(x_i+\frac12 \Big)-  & 4h^+(x_i)-2
\\
=2\Big (  h^+ \Big (x_i-\frac12 & \Big)- h^-(x_i) \Big)+2 \Big (h^+ \Big(x_i+\frac12 \Big)-h^+(x_i) \Big),
\end{aligned}
\end{equation}
and we observe that 
\begin{equation}\label{3l3}
\begin{gathered}
h^+\Big(x_i-\frac12\Big)-h^-(x_i)=\sharp  \Big( Y(X) \cap (x_i-\frac12,x_i) \Big),
\\
h^+\Big(x_i+\frac12\Big)-h^+(x_i)=-\sharp  \Big( Y(X) \cap (x_i,x_i+\frac12] \Big)\,.
\end{gathered}
\end{equation}
By \eqref{3l1}, \eqref{3l2} and \eqref{3l3} we conclude that
\begin{align} \label{3l5}
 &\lim_{\e\rightarrow0}\frac{4}{\e}\int_{x_i}^{x_i+\e}S^\e(x+\frac12)+S^\e(x-\frac12)-2S^\e(x)dx=-8\Delta(x_i),
\end{align}
where we have set 
\begin{align}
\Delta (x_i):=\sharp  \Big( Y(X) \cap (x_i,x_i+\frac12] \Big) 
- 
\sharp  \Big( Y(X) \cap (x_i-\frac12,x_i) \Big)\, .
\end{align}
Let us finally analyse the last line in \eqref{3lines}. To do this we first recall that 
$$\dot S^\e=\Big(-2\sum_{y\in Y(X), \, y\neq x_i} \delta_{y} \Big)- \delta_{x_i+\e}-\delta_{x_i},
$$ 
which yields,
for a.e. $x\in (x_i,x_i+\e)$,
\begin{align} \label{60}
\int_{x+\rho}^{x+\frac12}\frac{\dot S^\e(z)}{z-x}dz+\int_{x-\frac12}^{x-\rho}\frac{\dot S^\e(z)}{z-x}dz
= -2  \!\!\!\!\! \sum_{\substack{y\in (x+\rho,x+\frac12 ] \\ y \in Y(X) } }\frac{1}{y-x}   \quad-  2  \!\!\!\!\!  \sum_{\substack{y\in (x-\frac12,x-\rho) \\ y \in Y(X) }}\frac{1}{y-x} \,.
\end{align}
Notice that in the above sum, the terms containing $x_i+\e$ and $x_i$ do not appear because $x\in (x_i,x_i+\e)$. %
By integrating \eqref{60} with respect to $x \in (x_i,x_i+\e)$, we get that, for $\e$ small enough, the last line in \eqref{3lines} equals
\begin{equation} \label{3l4}
\begin{aligned}
&4\int_{x_i}^{x_i+\e}\sum_{\substack{y\in (x+\rho,x+\frac12] \\ y \in Y(X) } }\frac{1}{y-x}   + \sum_{\substack{y\in (x-\frac12,x-\rho) \\ y \in Y(X) }}\frac{1}{y-x} \, dx\\
&=-4\sum_{\substack{y\in (x_i+\rho,x_i+\frac12] \\ y \in Y(X) }}\log(|y-x_i-\e|)-\log(|y-x_i|)\\
&\;\;\;-4\sum_{\substack{y\in (x_i-\frac12,x_i-\rho) \\ y \in Y(X) }}\log(|y-x_i-\e|)-\log(|y-x_i|).
\end{aligned}
\end{equation}
Putting together \eqref{3l6}, \eqref{3lines}, \eqref{3l5}, \eqref{3l4} we infer
\begin{align}
&\lim_{\e\rightarrow0^+}\frac{1}{\e}(\mathcal E^N_{\rho}(x^\e_1,\dots,x^\e_N)-\mathcal E^N_{\rho}(x_1,\dots,x_N))\nonumber\\
&=-8\Delta(x_i)-4\sum_{  \substack{ y \in (x_i-\frac12,x_i)    \\ y \in Y(X) }  }\frac{1}{|y-x_i|}+4\sum_{ \substack{y\in (x_i,x_i+\frac12]   \\  y \in Y(X)   }}\frac{1}{|y-x_i|}.
\end{align}
where we have also used that $Y(X) \cap (x_i+\rho,x_i+\frac12] = Y(X) \cap  (x_i,x_i+\frac12] $ and also
$Y(X) \cap (x_i-\frac12,x_i - \rho) = Y(X) \cap  (x_i-\frac12,x_i ) $.

As anticipated, the computation in the case $\e<0$ is similar and yields the same limit, hence proving that the quantity above is the first variation of the energy $\mathcal E^N_{\rho}$.

Notice that 
$$
\Delta(x_i)= \sum_{\substack{ y \in (x_i ,x_i + \frac12]    \\ y \in Y(X)}}  \frac{y-x_i}{|y-x_i|}
+\sum_{    \substack{ y \in (x_i  -  \frac12, x_i)    \\ y \in Y(X)}}     \frac{y-x_i}{|y-x_i|}.
$$ 
We get
\begin{align}
\lim_{\e\rightarrow0}\frac{1}{\e} & (\mathcal E^N_{\rho}(x^\e_1,\dots,x^\e_N)-\mathcal E^N_{\rho}(x_1,\dots,x_N))\nonumber\\
&=4\Big(\sum_{\substack{ y \in (x_i - \frac12, x_i)    \\ y \in Y(X)}}\frac{-1 - 2(y -x_i)}{|y -x_i|}+\sum_{\substack{ y \in (x_i ,x_i + \frac12]    \\ y \in Y(X)}}\frac{1- 2(y -x_i)}{|y -x_i|}\Big). \label{kamehameha2}
\end{align}

\textit{Step 2. Rewriting $\mathcal E^N_{\rho}$.}
It is easy to check that the quantity at \eqref{kamehameha2} coincides with the partial derivative with respect to $x_i$ of the functional
\begin{align}\label{primitive}
 \tilde{\mathcal E}^N_{ \rho}(x_1,\dots,x_N) :=2\sum_{i=1}^N\Big(\sum_{\substack{ y \in (x_i-\frac12,x_i+\frac12] \\ y\neq x_i,  \, y\in Y(X) }}-\log(|y-x_i|)+2|y-x_i|\Big).
\end{align}
Notice that if  $ y \in (x_i-\frac12,x_i+\frac12]$, then $y= x_j + k$ for some $x_j\in X$ and $k\in\{-1,0,1\}$, and   $|y-x_i|= |x_j-x_i| \wedge (1- |x_j-x_i|)$. Then,  we can also write the functional $\tilde{\mathcal E}^N_{ \rho}$ in the equivalent way
\begin{align}
\tilde{\mathcal E}^N_{ \rho}(x_1,\dots,x_N)=&2\sum_{i\neq j}(-\log(|x_j-x_i|)\vee(-\log(1-|x_j-x_i|))\nonumber\\
&+4\sum_{i\neq j}|x_j-x_i|\wedge(1-|x_j-x_i|).
\end{align}
Finally, we will also make use of the following  formula
$$
\tilde{\mathcal E}^N_{ \rho}(x_1,\dots,x_N)=2\sum_{k=1}^{N-1}G_k(x_1,\dots,x_N),
$$
where
\begin{multline*}
 G_k(x_1,\dots,x_N):=
 \\
 \sum_{|i-j|\equiv k \;(\textrm{mod }N) }\big((-\log(|x_j-x_i|)\vee(-\log(1-|x_j-x_i|))\big)+2\big(|x_j-x_i|\wedge(1-|x_j-x_i|)\big) \, . 
\end{multline*}

\textit{Step 3. Minimization.}
Here we prove that $\tilde{\mathcal E}^N_{ \rho}$ is minimized when $x_1,\dots,x_N$ are evenly spaced on $\mathcal S^1$. To this purpose, we prove that such  configurations are the (unique)  minimizers of  $G_k$, for each $k=1,\dots,N-1$. First, we observe that the function
\begin{align}
(0,1)\ni y\mapsto f(y):=\big((-\log(|y|)\vee(-\log(1-|y|))\big)+2\big(|y|\wedge(1-|y|)\big)
\end{align}
is strictly convex. 

Now, without loss of generality we assume $x_1 < x_2 < \ldots < x_N$. Then, we fix $k\in\{1,\dots,N-1\}$ and for all $i=1,\dots,N$ we set

\begin{align}
d_{i}:=\begin{cases}
         |x_i-x_{i+k}|&\text{if }i+k\leq N,\\
         1-|x_i-x_{i+k-N}|&\text{if }i+k>N,
        \end{cases}
\end{align}
so that it turns out that 
\begin{align}\label{constraintk}
\sum_{i=1}^Nd_i=k, \qquad G_k(x_1,\dots,x_N)=\sum_{i=1}^Nf(d_i).
\end{align}
Therefore,  by Jensen inequality we deduce that 
$G_k$ is minimized if and only if $d_i=\frac{k}{N}$ for all $i=1,\dots,N$. This is achieved if and only if $\{x_i\}_{i=1}^N$ are evenly spaced (with respect to the distance $d$ on $\mathcal S^1$). 
The proof is achieved.

\section*{Acknowledgments}

Silvio Fanzon gratefully acknowledges support by the Christian Doppler
Research Association (CDG) and Austrian Science Fund (FWF) through the
Partnership in Research project PIR-27 ``Mathematical methods for
motion-aware medical imaging''. %
 We authors are members of  {\it Gruppo Nazionale
per l'Analisi Matematica, la Probabilit\`a e le loro Applicazioni} (GNAMPA) of the {\it Istituto Nazionale di
Alta Matematica} (INDAM).
 We thank the anonymous referee for the useful suggestions.

\bibliography{bibliography}

\bibliographystyle{my_plain}

\end{document}